\documentclass{article}
\usepackage{amssymb, amsfonts, amsmath, amsthm}
\usepackage{enumerate}
\usepackage{caption}
\usepackage[colorlinks=true,
citecolor=black,
linkcolor=black,
anchorcolor=black,
filecolor=black,
menucolor=black,
urlcolor=black
]{hyperref}
\usepackage[normalem]{ulem}
\usepackage{cancel}
\usepackage{array}
\usepackage{longtable}
\usepackage{tikz}
\usetikzlibrary{calc}   
\usepackage{epsfig,lpic,wrapfig}
\usepackage{bm}
\usepackage{multicol}

\usepackage[utf8]{inputenc}
\usepackage[T1,T2A]{fontenc}
\usepackage[russian,
german,
english]{babel}

\usepackage{csquotes}
\usepackage{tikz}
\usepackage{tikz-3dplot}
\usetikzlibrary{patterns} 

\def\thetitle{Ratio of Intrinsic Metric to Extrinsic Metric and Volume}
\def\theauthors{Berk Ceylan}
\hypersetup{pdftitle={\thetitle},
pdfauthor={\theauthors}}

\newcommand{\Addresses}{{\bigskip\footnotesize

\medskip

\noindent   Berk Ceylan, 
\par\nopagebreak
 \textsc{}
  \par\nopagebreak
  \textit{Email}: \texttt{berk.ceylan@epfl.ch}
  
}}

\theoremstyle{theorem}
\newtheorem{theorem}{Theorem}
\newtheorem{lemma}{Lemma}

\usepackage{sectsty}

\sectionfont{\fontsize{10}{15}\selectfont}

\newcounter{thm}[section]

\def\claim#1{\par\medskip\noindent\refstepcounter{thm}\hbox{\bf\boldmath #1.}
\it\ 
}
\def\endclaim{
\par\medskip}

\newcommand*{\z}[1]{#1\nobreak\discretionary{}%
            {\hbox{$\mathsurround=0pt #1$}}{}}

\begin{document}

\title{\thetitle}
\author{\theauthors}

\date{}
\maketitle

\begin{abstract}
    We study the relationship between the ratio of intrinsic to extrinsic metrics and area. For certain surfaces inside unit ball in $\mathbf{R}^3$ we give lower bound on the maximum of ratio in terms of its area. We also give examples to show non-existence of global lower bounds.
\end{abstract}
\begin{center}
    \section{Introduction}
\end{center}
Let $X\subset Y$ be intrinsic metric spaces and $X$ has the induced metric. Denote by $|xy|_X,|xy|_Y$ the metrics of $X,Y$ respectively. We define the following function which roughly measures how much $X$ is curved inside $Y$:

$K:X\times X \rightarrow\mathbf{R}$, $K(x,y)$ =$
\left\{
	\begin{array}{ll}
		1  & \mbox{if } x =y \\
		\frac{|xy|_X}{|xy|_Y} & \mbox{if } x \neq y
	\end{array}
\right.$.

Observe that $K$ is a continuous function and $K\geq 1$. Our objective is to investigate the relation of this function with the volume of $X$. Let us denote by $K_{max}$ the maximum of $K$. We investigate when it is possible to obtain a lower bound on $K_{max}$ with a lower bound on the area of a surface inside unit ball in $\mathbf{R}^3$. Note also that $K_{max}$ is equal to minimal Lipschitz constant.

For the case of convex surfaces inside unit ball, $K$ is studied at \cite{MR28052}, see also \cite{alexandrovSelectedWorksII}.

Our main objective is to find a general class of surfaces $X\subset B(1) \subset \mathbf{R}^3$ such that by bounding the area from below we can bound $K_{max}$ from below too. Moreover we wish to obtain a lower bound on $K_{max}$ such that as area lower bound goes to infinity $K_{max}$ does so too. It is because of this we need to make assumptions on $X$. As we show in section 3 that there are families of surfaces with $K\leq 2$ and arbitrarily large area. 

Before focusing on surfaces in unit ball, we work in more general setting of arbitrary dimensions. We essentially compare packing numbers of spaces involved to obtain lower bound on $K_{max}$. To obtain packing number comparison we use lower bounds on volume. However to make this approach uniform we have to put additional constraints. One way is to bound the volume of balls in our space by  a constant.  Then on the class of spaces for which volume of balls of certain radius is bounded above by a certain constant, we can obtain the desired relation between $K_{max}$ and volume. With the notation $diam(X)$ for intrinsic diameter of $X$, $a_{X}(r,m)=\sup\limits_{x\in X} \mathcal{H}_m(B(x,r))$ where $\mathcal{H}_m$ denotes $m$-dimensional Hausdorff measure and finally $\beta_Y(r)$ denotes $r$-packing number of $Y$. Then we have:

\begin{theorem}
        Let $X\subset B^n_1$ be an intrinsic metric space for which metric is induced by $B^n_1$, where $B^n_1$ denotes unit ball in $\mathbf{R}^n$. Assume $diam(X) \geq J_n$ and $a_{X}(\frac{J_n}{3},m) \leq C(\frac{J_n}{3},m)$. Then if $\mathcal{H}_m(X)>C(\frac{J_n}{3},m)\cdot \beta_{B^n_1}(\delta)$ we have $K\geq \frac{J_n}{3\delta}$ for some pair of points on $X$.
\end{theorem}
    
Bounding diameter from below by a small constant $J_n$, which is explained in section 2, does not effect the generality of statement. From the theorem it follows from Bishop-Gromov inequality, see \cite{Chavel_2006}, that one can put lower bound on minimal Lipschitz constant of embeddings of Riemannian manifolds with Ricci curvature bounded below in terms of its volume. 

Then in later sections we use more geometric approach to problem in the case that $X$ is a surface in $\mathbf{R}^3$. We first prove an analogous theorem for spaces which we call to be triangular as it is the boundary case for both area and intrinsic metric. Then we extend it to spaces for which admit a triangular skeleton of required property.

In the last section we explain briefly the quantitative study of the number of essentially different pairs with $K$ large.

\begin{centering}
    \section{Spaces With Bounded Balls Yet Large Volume}
\end{centering}

Let $X\subset B^n_1$ be an compact intrinsic metric space with induced metric from $B^n_1$ the unit ball in $\mathbf{R}^n$. As we are interested in the function $K$ which is scale invariant we can assume that $diam(X)\geq \sqrt{\frac{2(n+1)}{n}}=:J_n$. To see this assume $diam(X) < J_n$. Then as extrinsic diameter of $X$ (denoted $diam_{ext}(X)$) is smaller than intrinsic one we have $diam_{ext}(X) < J_n$ but from Jung's Theorem \cite{Jung1901}, we have if $r$ is the radius of minimal enclosing ball of $X$ we get $r\leq diam_{ext}(X) \cdot \frac{1}{J_n}<1$. So $X$ lies on a smaller ball. Hence we can scale up at least until $diam(X)\geq J_n$. Assume from now on this is the case. Then it makes sense to talk about $a_X(r,m):=\sup\limits_{x\in X}\mathcal{H}_m(B_X(x,r))$ for any $r \leq \frac{J_n}{2}$, where $\mathcal{H}_m$ is the Hausdorff measure of dimension $m$. Let us denote by $\beta_Y(\delta)$ the $\delta$-packing number of $Y$. Then we have our key lemma:

\begin{lemma}
    Assume $X$ as above with $dim_H(X)=m$. Then if $\mathcal{H}_m(X)> a_{X}(r,m)\cdot \beta_{B^n_1}(\delta)$, we have $K(x,y)\geq \frac{r}{\delta}$ for some pair of points $x,y \in X$.
\end{lemma}
\begin{proof}
    Let $x_i\in X$ , $i=1,...,m$ such that each pair is at least $r$ away from each other and $m$ is maximal. Therefore the balls $B_i := B_X(x_i,r)$ cover $X$. Hence we have $m \cdot a_{X}(r,m) \geq \mathcal{H}_m(X) > a_{X}(r,m) \beta_{B^n_1}(\delta) $. Then $m> \beta_{B^n_1}(\delta)$, so among $x_i$ there is a pair with $|x_ix_j|_{B^n_1} \leq \delta$. Since $x_j \notin B_i$ we have $|x_ix_j|_X\geq r$. Thus $K(x_i,x_j)= \frac{|x_ix_j|_X}{|x_ix_j|_{B^n_1}}\geq \frac{r}{\delta}$.
\end{proof}

Let $C(\frac{J_n}{3},m) >0$ be a constant e.g. it makes sense to put volume of a ball of radius $\frac{J_n}{3}$ in $\mathbf{R}^m$. We have the following theorem mentioned in the introduction whose proof follows from Lemma 1:

\begin{theorem}
    Let $X\subset B^n_1$ be an intrinsic metric space for which metric is induced by $B^n_1$, where $B^n_1$ denotes unit ball in $\mathbf{R}^n$. Assume $diam(X) \geq J_n$ and $a_{X}(\frac{J_n}{3},m) \leq C(\frac{J_n}{3},m)$. Then if $\mathcal{H}_m(X)>C(\frac{J_n}{3},m)\cdot \beta_{B^n_1}(\delta)$ we have $K\geq \frac{J_n}{3\delta}$ for some pair of points on $X$.
\end{theorem}

The important point here is that lower bound on $K$ only depends $\delta$. And as $\delta$ approaches $0$, lower bound on $K$ diverges. Which roughly says that $X$ has to "curve" increasingly more to have a large volume.

\begin{centering}
    \section{Families of Surfaces With Bounded $K$ But Diverging Volume}
\end{centering}

Phenomena that lower bound on $K_{max}$ diverging is not true without additional assumptions on $X$ as is seen by following example which is communicated to us by Alexander Lytchak. 

Let $I^3$ denote the unit cube. Uniformly divide $I^3$ into smaller cubes of side length $\frac{1}{N}$, for $N \geq 1$ an integer. Then consider union of $0$-skeleton and $1$-skeleton of the partition. At each vertex put a square of side length $(\frac{1}{N})^{5/4}$ parallel to one of the faces of $I^3$. Define $X_N$ to be union of small squares and $1$-skeleton, see Figure 1. We have $area(X_N)=N^3\cdot (\frac{1}{N})^{5/2}=\sqrt{N}$. It is a simple exercise to check that $K$ is bounded on $X_N$, say by 2. 

Note that one can thicken $1$-skeleton and arrange its position so that $X_N$ becomes a smooth surface.

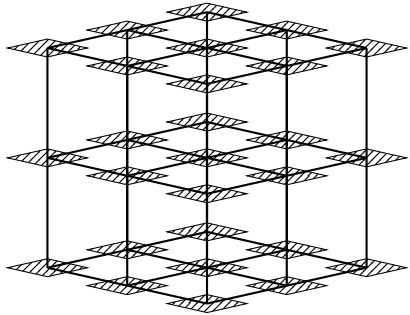
\begin{figure}[h]
    \centering
    \tdplotsetmaincoords{77}{135}
\begin{tikzpicture}[tdplot_main_coords, scale=3]

\def\s{1/4} 

\foreach \x in {0,0.5,1} {
    \foreach \y in {0,0.5,1} {
        \foreach \z in {0,0.5,1} {
            \draw[pattern=north east lines, pattern color=black] 
                (\x-\s/2,\y-\s/2,\z) 
                -- (\x+\s/2,\y-\s/2,\z) 
                -- (\x+\s/2,\y+\s/2,\z) 
                -- (\x-\s/2,\y+\s/2,\z) 
                -- cycle;
        }
    }
}

\foreach \x in {0,0.5} {
    \foreach \y in {0,0.5,1} {
        \foreach \z in {0,0.5,1} {
            \draw[black, thick] (\x,\y,\z) -- (\x+0.5,\y,\z);
        }
    }
}

\foreach \x in {0,0.5,1} {
    \foreach \y in {0,0.5} {
        \foreach \z in {0,0.5,1} {
            \draw[black, thick] (\x,\y,\z) -- (\x,\y+0.5,\z);
        }
    }
}

\foreach \x in {0,0.5,1} {
    \foreach \y in {0,0.5,1} {
        \foreach \z in {0,0.5} {
            \draw[black, thick] (\x,\y,\z) -- (\x,\y,\z+0.5);
        }
    }
}

\end{tikzpicture}
    \caption{Picture of $X_2$.}
    \label{fig:placeholder}
\end{figure}

There are also families of examples with fixed genus. Reader can take hint from the Figure 1. and produce them.

\begin{centering}
    \section{Triangular Spaces}     
\end{centering}

Let $x\in\mathbf{R}^3$ and $y_1,...,y_k\in \partial B(x,R)$ where $R >0$ is arbitrary. Let $X=\cup_{i=1}^{k} \Delta(y_ixy_{i+1})$, where $\Delta(y_ixy_{i+1})$ denotes the triangle determined by $x,y_i,y_{i+1}$ and $y_{k+1}=y_1$. We additionally assume that $X$ is a topological manifold to not allow self-intersections. Let us call such $X$ \textit{triangular}. We will find a lower bound on the area for triangular $X$ such that there will be a pair of points with $K\geq \frac{R}{\delta}$ where $\delta$ depends on the area lower bound. First let us investigate geometry of triangular spaces.

\begin{figure}[h!]
    \centering
    
    \begin{tikzpicture}[scale=3, line cap=round, line join=round]

\coordinate (A) at (0.11, 0.01);      
\coordinate (B) at (-1.07, -0.05);
\coordinate (C) at (-0.88, -0.57);
\coordinate (D) at (-1.15, -0.39);
\coordinate (E) at (-0.72, -0.19);
\coordinate (F) at (-0.38, -0.58);
\coordinate (G) at (0.05, -0.27);
\coordinate (H) at (0.40, -0.56);
\coordinate (I) at (0.75, -0.10);
\coordinate (J) at (1.05, -0.41);
\coordinate (K) at (1.08, 0.04);

\draw[very thick]
  (C)--(E)--(F)--(G)--(H)--(I)--(J)--(K);

  \draw[very thick]
  (B)--(D)
  (B)--(C);

\foreach \p in {B,C,D,E,F,G,H,I,J,K}{
  \draw[thin,gray!60] (A)--(\p);
}

\draw[very thick]
  
  (A)--(E)
  (A)--(F)
  (A)--(G)
  (A)--(H)
  (A)--(I)
  (A)--(B)
  
  (A)--(K);


\end{tikzpicture}
    \caption{A triangular space.}
    \label{fig:placeholder}
\end{figure}
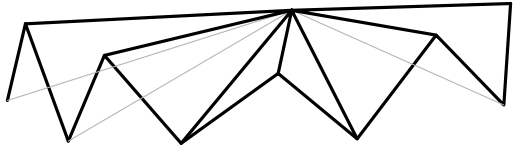

\begin{lemma}
    Let $X$ triangular with vertices, $x,y_1,...,y_{k}$. Let $\sigma:y_i\rightarrow y_j$ be a minimal path joining $y_i$ to $y_j$. Let $\psi(i,j)$ be the sum of angles at $x$ of all the faces $\sigma$ passes through. Then if $\psi(i,j) \geq \dfrac{\pi}{3}$ we have $|y_iy_j|\geq R$.
\end{lemma}

\begin{proof}
    Let us denote by $\Delta_{i_1},...,\Delta_{i_n}$ the faces of $X$ that $\sigma$ passes through. Let us also denote by $\theta_{i_1},...,\theta_{i_n}$ the angles of $\Delta_{i_l}$'s at the vertex $x$. Assume that $\psi(i,j)=\sum_l (\theta_{i_l)}\leq 2\pi$. Then on the plane pick a point $\bar{x}$ and on the circle of radius $R$ around $\bar{x}$ pick points $\bar{p}_l$ for $l=1,...,n$ such that they go in counter-clockwise order and $\angle(\bar{p}_l,x,\bar{p}_{l+1})=\theta_{i_l}$ for $l=1,...,n-1$. If  $X'$ is the union of the triangles $\Delta(\bar{p}_l,x,\bar{p}_{l+1})$ considered with the intrinsic metric, then $|\bar{p}_1\bar{p}_n|_{X'}=|y_iy_j|_X$. Observe if $\psi(i,j)=\sum_l (\theta_{i_l})\geq \pi$ then $|\bar{p}_1\bar{p}_n|_{X'}=2R$ as going from $\bar{p}_1$ to $\bar{x}$ and to $\bar{p}_n$ is the shortest path. So assume $\psi(i,j)< \pi$. From chord length formula we have $|\bar{p}_1\bar{p}_n|_{X'}=2R\sin({\dfrac{\psi(i,j)}{2}})$. We have $2R\sin({\dfrac{\psi(i,j)}{2}}) \geq R$ if and only if $\sin({\dfrac{\psi(i,j)}{2}})\geq \dfrac{1}{2}$ so that $\psi(i,j)\geq\dfrac{\pi}{3}$.

    If $\psi(i,j)=\sum_l (\theta_{i_l})> 2\pi$ then restrict the triangles until sum of the angles is at most $2\pi$ and observe that $|y_iy_j|_X=2R$.
\end{proof}

\begin{figure}[h!]
    \centering
    
    \begin{tikzpicture}[scale=2.4, line cap=round, line join=round]

\def\r{1}        
\def\n{8}        
\def\angleSpan{140} 

\coordinate (X) at (0,0);

\draw[thick] (X) ++(90:\r) arc[start angle=90, end angle=-50, radius=\r];

\foreach \i in {1,...,\n}{
  \coordinate (P\i) at ({\r*cos(90 - (\i-1)*\angleSpan/(\n-1))},
                        {\r*sin(90 - (\i-1)*\angleSpan/(\n-1))});
}

\foreach \i in {1,...,\n}{
  \draw[thick] (X) -- (P\i);
}

\foreach \i [evaluate=\i as \next using int(\i+1)] in {1,...,\n}{
  \ifnum\next<\numexpr\n+1
    \draw[thick] (P\i) -- (P\next);
  \fi
}

\draw[dashed, thick, gray] (P1) -- (P\n);

\foreach \i in {1,...,\n}{
  \fill (P\i) circle (0.015);
}
\fill (X) circle (0.015);

\node[below left=-1pt] at (X) {$\bar{x}$};
\node[above] at (P1) {$\bar{p}_1$};
\node[below right=-1pt] at (P\n) {$\bar{p}_l$};

\end{tikzpicture}
    \caption{Chord in the development is geodesic.}
    \label{fig:placeholder}
\end{figure}
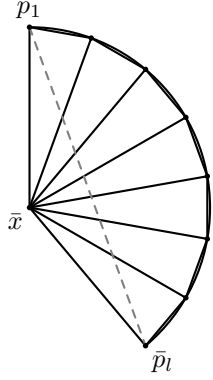

\begin{lemma}
    Let $X$ triangular with vertices, $x,y_1,...,y_{k}$. Then we have 
    
    $\dfrac{k}{2}\geq \dfrac{c(R)}{R^2}$.
\end{lemma}

\begin{proof}
First observe that for each triangle $\Delta_i=\Delta(y_i,x,y_{i+1})$ we have 

$area(\Delta_i)\leq \dfrac{R^2}{2}$. To see that we use Heron's formula with $\angle(y_i,x,y_{i+1})=\theta_i$:

\begin{alignat*}{5}
area(\Delta_i)&=\dfrac{R\cdot sin(\dfrac{\theta_i}{2})}{2}\cdot \sqrt{4R^2-4R^2sin(\dfrac{\theta_i}{2})^2}\\
&=R^2sin(\dfrac{\theta_i}{2})cos(\dfrac{\theta_i}{2})\\
&=R^2sin(\theta_i)\\
&\leq \dfrac{R^2}{2}.
\end{alignat*}

Then we have, $c(R) \leq area(X)=\sum_i area(\Delta_i)\leq k\cdot \dfrac{R^2}{2}$.
     
\end{proof}
\begin{theorem}
    Let $X$ triangular with vertices, $x,y_1,...,y_{k}$. If $area(X)> c(R)$ then there is a pair among $\{y_i\}$ such that $K\geq \dfrac{R}{\delta}$.
\end{theorem}
\begin{proof}
   Put $a_1=y_1.$ Inductively given $a_i=y_j$ we put $a_{i+1}=y_l$ such that $l>j+1$, $\psi(j,l)\geq \dfrac{\pi}{3}$ and $\psi(l,1) \geq \dfrac{\pi}{3}$. From the lower bound on area observe that with the triangles in $X$ we can cover a disc of radius $R$ at least $\beta_{B(R)}(\delta)$ times. Hence if selection $a_1,...,a_N$ maximal, then $N>\beta_{B(R)}(\delta)$. Thus there is a pair such that $|a_ia_j|_{B(R)}\leq\delta$ and as $|a_ia_j|_X\geq R$ from Lemma 3 we have $K(a_i,a_j)\geq \dfrac{R}{\delta}.$
\end{proof}
Now we generalize above theorem to spaces such that $X=\bigcup\limits_{i=1}^{k} \Delta(y_ixy_{i+1})$ with $|y_ix|\leq R$. Let us call such a space \textit{R-triangular}.
\begin{lemma}
    Let $X$ be a $R-$triangular space with vertex $x$. If $area(X) > c(R)$ then there exists a pair of points such that $K \geq \dfrac{R}{\delta}.$
\end{lemma}
\begin{proof}
    Let $y_1...,y_k$ be vertices of $X$. For each $i$ pick $\bar{y_i}$ on the ray $[x,y_i)$ such that $|\bar{y_i}x|=R$. Then $\bar{X}=\bigcup\limits_{i=1}^{k} \Delta(\bar{y_i}x\bar{y}_{i+1})$ satisfy conditions of Theorem 2. Then there is a pair of points on $\bar{X}$ with $\bar{K}\geq \dfrac{R}{\delta}$ where $\bar{K}$ is the corresponding ratio function on $\bar{X}$. Observe that the ratio function $\bar{K}$ is scale invariant about $x$. Hence if we pick $\lambda >0$ such that $\lambda \bar{X} \subset X$ we find a pair on $X$ with required property where $\lambda \bar{X}$ is $\bar{X}$ shrunk down towards $x$ by a factor of $\lambda$.
\end{proof}

Observe that if $P$ is a polyhedron and $x\in P$ a vertex, then Lemma 5 can be applied to union of all faces containing $x$ as it can be triangulated to be an $R-$triangular space.

Let $x,y_1,y_2,y_3$  be distinct points such that $|xy_i|_{\mathbf{R}^3}=R$ for $i=1,2,3$. Consider the union of triangles $\Delta(y_1xy_2)$ and $\Delta(y_2xy_3)$. Denote it by $X$. Observe that if we replace each triangle with a surface of same boundary in a non-self-intersecting way to obtain $\bar{X}$ we have $area(X)\leq area(\bar{X})$ and $|y_iy_j|_X\leq |y_iy_j|_{\bar{X}}$. Thus we can generalize theorem 2 to such surfaces. This is made precise in the following:

\begin{theorem}
    Let $X$ be an $R$-triangular space with vertex $x$ such that $area(X) > c(R)$. Assume that $\bar{X}=B(x,r)$ where $r\leq R$ is a surface such that $1$-skeleton of $X$ isometrically embeds with $x$ fixed and boundary goes to boundary. Then there is a pair such that $K_{\bar{X}}\geq \dfrac{R}{\delta}$.
\end{theorem}

\begin{centering}
    \section{Counting Geodesics With Large Distortion}
\end{centering}

We wish give a lower bound on number of pairs of points such that $K\geq \frac{J_n}{3\delta}$. But we also would like to neglect the pairs appearing because of the continuity of $K$. For this reason we want pairs such that they are sufficiently far away from each other in $X\times X$. We can do this by increasing lower bound on volume.

In the statements of Theorem 2 and Theorem 3 assume we replace packing number term in lower bound of volume by $\beta+N$ where $N >0$ and $\beta$ is any of the packing number depending on the context. Then on the proofs the number of chosen points would become $N$ more. Which would mean we can keep picking pairs until number of points is at least $\beta$. This amounts in $N+1$ pairs. And if we do the picking in such a way that we forget one of the points of previous pair at each step, our obtained $N+1$ pairs have the property that on $X\times X$ each is at least far away from each other by a fixed constant.

\bibliographystyle{plain} 
\bibliography{references} 

\Addresses
\end{document}